\documentclass[a4paper]{article}[11pt]
\usepackage[top=3cm, bottom=2.5cm, left=2cm, right=2cm]{geometry}
\usepackage{amsmath,amssymb,amsthm}
\usepackage{setspace}
\usepackage{enumerate}
\usepackage[pdftex]{graphicx}
\graphicspath{{fig/}}
\usepackage[justification=centering]{caption}
\usepackage{subcaption}
\DeclareGraphicsExtensions{.pdf,.jpg,.png,.bmp}

\newtheorem{theorem}{\bf Theorem}[section]

\newtheorem{remark}[theorem]{\bf Remark}
\begin{document}

\title{Fast algorithm of adaptive Fourier series}
\author{You Gao$^1$, Min Ku $^{2}$\footnote{Corresponding author.},
Tao Qian$^{3}$
\thanks{E-mail: gaoyou@mail.sysu.edu.cn, kumin0844@163.com, fsttq@umac.mo.
} \\
$^1$School of Mathematics (Zhuhai), Sun Yat-Sen University.\\
$^2$CIDMA, Department of Mathematics, University of Aveiro.\\
$^3$Department of Mathematics, University of Macau.}

%

\date{}
\maketitle
\begin{abstract}
\noindent Adaptive Fourier decomposition (AFD, precisely 1-D AFD or Core-AFD) was originated for the goal of positive frequency representations of signals. It achieved the goal and at the same time offered fast decompositions of signals. There then arose several types of AFDs. AFD merged with the greedy algorithm idea, and in particular, motivated the so-called pre-orthogonal greedy algorithm (Pre-OGA) that was proven to be the most efficient greedy algorithm. The cost of the advantages of the AFD type decompositions is, however, the high computational complexity due to the involvement of maximal selections of the dictionary parameters. The present paper offers one formulation of the 1-D AFD algorithm by building the FFT algorithm into it. Accordingly, the algorithm complexity is reduced, from the original $\mathcal{O}(M N^2)$ to $\mathcal{O}(M N\log_2 N)$, where $N$ denotes the number of the discretization points on the unit circle and $M$ denotes the number of points in $[0,1)$. This greatly enhances the applicability of AFD. Experiments are carried out to show the high efficiency of the proposed algorithm.
\end{abstract}



\section{Introduction}\label{Intro}

\noindent One-dimensional adaptive Fourier decomposition (abbreviated as 1-D AFD) has recently been proposed and proved to be among the most effective greedy algorithms \cite{TY, Qian2014b, TQ}. AFD is originally developed for the contexts of the unit disc and the upper-half plane and now is formally called 1-D AFD, or Core-AFD. The reason for the last terminology is due to the fact that it becomes the constructive block of the lately developed variations of 1-D AFD, such as Unwinding AFD and Cyclic AFD, where the former is an algorithm for more effective frequency decompositions of signals \cite{T1}, and the latter is for finding solutions of $n$-best rational approximations of functions in the Hardy space \cite{QWe, Q22}. Most recently, the concept of AFD is generalized to approximations of linear combinations of the Szeg\"o kernels and their derivatives. In the later studies, such approximations are not necessarily obtained through greedy-type algorithms, viz., the maximal selection principle \cite{GS,RA}. They can be obtained by any method, for instance, SVM in learning theory \cite{MoQ1}, the regularizations in compressed sensing \cite{LQM1}, or the Tikhonov regularization, etc. In this article, we focus on a AFD algorithm of the greedy type using the maximal selection principle. The AFD type decompositions all have promising applications to system identification and signal analysis \cite{TM, DQ2} with proven effectiveness. Recently, 1-D AFD has been generalized in either the Clifford (Quaternionic) algebra \cite{QSW, QWY}, or the several complex variables settings \cite{Qian2014b}. In the sequel, when we use the notion AFD, we will specify the context for clearness.

\noindent However, since 1-D AFD involves maximal selections of the parameters in the Szeg\"{o} kernels, it has great computational complexity. For instance, the algorithm in \cite{TLZ} is shown to be of the computational complexity $\mathcal{O}(M N^2)$, where $N$ is the discretization of the unit circle and $M$ is the number of samples in the radius of the unit disc, on which the maximal value is selected. We note that the quantity $M$ can not be reduced because it is independent on the discretization on the unit circle. On the one hand, it has the necessity to reduce the computational complexity in order to make AFD more practical in applications. On the other hand, it is, in fact, feasible to build in FFT into the AFD algorithm. In the present paper, we provide one such algorithm reducing the complexity to $\mathcal{O}(M N\log_2 N)$ from the original $\mathcal{O}(M N^2)$, where $N$ is for the discretization of the unit circle and $M$ is the number of samples in the radius of the unit disc. In Section 2, we briefly recall 1-D AFD and its discretization scheme. In Section 3 we introduce the proposed algorithm and analyse its computational complexity. In Section 4, we give numerical examples to compare the precisions, the selected parameters and the related errors between what we propose with the original 1-D AFD algorithm.

\section{Preliminaries}

\noindent Let $L^2$ be the Hilbert space of signals with finite energy on the closed interval $[0,2\pi]$, equipped with the inner product
\begin{eqnarray}
\left<G,F\right> = \int^{2\pi}_{0}G(e^{it})\overline{F}(e^{it})dt,
\end{eqnarray}
where $G,F:[0,2\pi]\rightarrow\mathbb{C}$, and $\overline{a}$ denotes the usual complex conjugate of $a\in\mathbb{C}$ \cite{GD}.
$H^2 = H^2(\mathbb{D})$ denotes the Hardy space on the unit disk $\mathbb{D} = \left\{z\in\mathbb{C}: |z|<1\right\}$ of the complex plane $\mathbb{C}$.
$\{B_k\}^{+\infty}_{k=1}$ is the Takenaka-Malmquist system or orthonormal rational function system, see Ref. e.g. \cite{JG,BA,AH,BB}, where
\begin{eqnarray}
B_k(z) = B_{a_1,a_2,\ldots,a_k}(z) = \frac{1}{\sqrt{2\pi}}\frac{\sqrt{1-|a_k|^2}}{1-\overline{a}_k z}\prod\limits^{k-1}_{l=1}\frac{z-a_l}{1-\overline{a}_l z},
\end{eqnarray}
$a_k\in \mathbb{D},k\in\mathbb{N}$.

\noindent Based on $\{B_k\}^{+\infty}_{k=1}$, the core algorithm of AFD is constructed $\cite{TY, T1, TLZ}$. For a given analytic signal $G\in H^2$, with $G_1 = G$, there exits the decomposition
\begin{eqnarray}
G(z) = \left<G_1, e_{a_1}\right>e_{a_1} + G_2(z)\frac{z-a_1}{1-\overline{a_1}z},
\end{eqnarray}
where $e_{a_1} = \frac{\sqrt{1-|a_1|^2}}{1-\overline{a}_1 z}, a_1 \in \mathbb{D}, z \in \partial \mathbb{D}$,
\begin{eqnarray}
G_2(z) = \left(G_1 - \left<G_1, e_{a_1}\right>e_{a_1}(z)\right)\frac{1-\overline{a_1}z}{z-a_1},
\end{eqnarray}
and $a_1\in \mathbb{D}$ is selected by according to the maximal selection principle.
That is,
\begin{eqnarray}
a_1 = \arg\max\limits_{a\in\mathbb{D}}\left\{\left|\left<G_1,e_a\right>\right|^2\right\},
\end{eqnarray}
which is crucial for the core algorithm of AFD (cf. e.g. $\cite{TY}$). Repeating such process to the $n$-th step, we get
\begin{eqnarray}\label{Sn}
G(z) = \sum\limits^n_{k=1}\left<G_k,e_{a_k}\right>B_{a_1,\ldots,a_k}(z) + G_{n+1}(z)\prod\limits^n_{k=1}\frac{z-a_k}{1-\overline{a_k}z}=S_n+ G_{n+1}(z)\prod\limits^n_{k=1}\frac{z-a_k}{1-\overline{a_k}z},
\end{eqnarray}
where the reduced reminder $G_{k+1}$ is obtained through the recursive formula
\begin{eqnarray}
G_{k+1}(z) = \left(G_k(z)-\left<G_k,e_{a_k}\right>e_{a_k}\right)\frac{1-\overline{a_k}z}{z-a_k},
\end{eqnarray}
and
\begin{eqnarray}\label{ak}
a_k = \arg\max\limits_{a\in\mathbb{D}}\left\{\left|\left<G_k,e_a\right>\right|^2\right\}.
\end{eqnarray}

\noindent Moreover, due to the orthogonality and the unimodular property of Blaschke products, for each $n$ there holds

\begin{eqnarray}
\left\| G - \sum\limits^n_{k=1}\left<G_k,e_{a_k}\right>B_{k} \right\|^2 = \|G\|^2 - \sum\limits^n_{k=1}\left|\left<G_k,e_{a_k}\right>\right|^2.
\end{eqnarray}
As a M\"{o}bius transform is of norm 1 on the unit circle, the above equation is equal to $\left\|G_{n+1}\right\|^2$. The equation $\langle G_k,e_{a_k}\rangle=\langle G,B_k\rangle$ also holds because of the orthogonalization of $\{B_k\}_{k=1}^n$.
Then
\begin{equation}
\lim\limits_{n\rightarrow +\infty} \left\|G_{n+1}\right\|^2 =\lim\limits_{n\rightarrow +\infty}\left\| G - \sum\limits^n_{k=1}\left<G_k,e_{a_k}\right>B_{k} \right\|^2 =\lim\limits_{n\rightarrow +\infty} \left\| G - \sum\limits^n_{k=1}\left<G,B_k\right>B_{k}\right\|^2.
\end{equation}
It has been proved that above limits convergent to zero when parameters $a_k$ is selected by (\ref{ak}) (cf. e.g. \cite{TY}). Finally we have $$ G= \sum\limits^\infty_{k=1}\left<G_k,e_{a_k}\right>B_{k}=\sum\limits^\infty_{k=1}\left<G,B_k\right>B_{k}.$$

\noindent The above is called 1-D AFD or Core AFD. It can be shown that for $f(z)\in H^2(\mathbb{D})$, as $z\to e^{it}$ in the non-tangential manner there exists the non-tangential limits $f(e^{it})$ for almost all $e^{it}\in \partial \mathbb{D}$ (cf. e.g. \cite{JG}). The mapping between $f(z)$ and its boundary limit function $f(e^{it})$ is an isometric isomorphism. Then $f(e^{it})$ could be processed by the approximation theory in $H^2(\mathbb{D})$. The biggest computation amount in 1-D AFD of $f(e^{it})$ is to find a point $a_k\in\mathbb{D}, k=1,2,\cdots,$ satisfying
$\left|\left<G_k, e_{a_k}\right>\right|^2 = \max\limits_{a\in\mathbb{D}}\left|\left<G_k, e_{a}\right>\right|^2,$ where $e_{a} = \frac{\sqrt{1-|a|^2}}{1-\overline{a} z}, a\in\mathbb{D},z\in\partial\mathbb{D}$. The key step is to compute the following integral
\begin{eqnarray}\label{formula2}
\left< G_k, e_{a}\right> = \frac{1}{2\pi}\int^{2\pi}_0 G_k(e^{it}) \frac{\sqrt{1-|a|^2}}{1-a e^{-it}}dt, \forall a\in \mathbb{D}.
\end{eqnarray}

\section{Formulation, algorithm and complexity analysis}

\noindent In this section we will derive our approximation procedure of $(\ref{formula2})$ incorporating FFT. We denote $G_k$ as $G$ for simplifying our notation.

\subsection{Formulation}

\noindent We first give a discrete numerical model of $(\ref{formula2})$. Suppose that the interval $[0,2\pi)$ is evenly divided into $0 = t_0 < \dots<t_m<\dots < t_{2^K -1} < 2\pi$ , $t_m=\frac{2\pi m}{2^K}$ with $K$ being large. Then
\begin{eqnarray}\label{formula1}
\left< G, e_{a} \right> = \frac{1}{2\pi}\int^{2\pi}_0 G(e^{it}) \frac{\sqrt{1-|a|^2}}{1-a e^{-it}}dt \approx \sum\limits^{2^K-1}_{m=0} \frac{\sqrt{1-|a|^2}}{2^K} G\left(e^{i\frac{2\pi m}{2^K}}\right) \frac{1}{1-a e^{-i \frac{2\pi m}{2^K}}}.
\end{eqnarray}

\noindent We index $a\in \mathbb{D}$ in polar coordinate. For a fixed $r\left(0 < r <1\right)$, the circle of radius $r$ is evenly divided by the $2^K$ points $e^{i\frac{2\pi}{2^K}j}, j =0,1,\ldots, N= 2^K-1$, which is the same segmentation step distance as $(\ref{formula1})$. That is $a_j = r e^{i\frac{2\pi}{2^K}j}, j =0,1,\ldots, 2^K-1$. The relation (\ref{formula1}), therefore, can be rewritten by approximating $a$ by $a_j$ on a fixed circle as
\begin{eqnarray}\label{ku1}
\left< G, e_{a_j} \right> \approx \sum\limits^{2^K-1}_{m=0} \frac{\sqrt{1- r^2}}{2^K} G\left(e^{i\frac{2\pi m}{2^K}}\right) \frac{1}{1-a_j e^{-i \frac{2\pi m}{2^K}}}, j = 0,1,2,\ldots, 2^K-1.
\end{eqnarray}

\noindent To further reduce (12), we define the notation for the right side of $(\ref{ku1})$ as
\begin{equation}\label{term81}
\langle G,e_{a_j}\rangle^{\tilde{}}= \sum\limits^{2^K-1}_{m=0} \frac{\sqrt{1- r^2}}{2^K} G\left(e^{i\frac{2\pi m}{2^K}}\right) \frac{1}{1-a_j e^{-i \frac{2\pi m}{2^K}}}, j = 0,1,2,\ldots, 2^K-1.
\end{equation}
Then, a simple computation gives
\begin{eqnarray}\label{ku12}
\langle G,e_{a_j}\rangle^{\tilde{}}= \sum\limits^{2^K-1}_{l=0}\frac{\sqrt{1- r^2}}{2\pi}\frac{a_j^l}{1-a_j^{2^K}} c_l, j = 0,1,2,\ldots, 2^K-1,
\end{eqnarray}
where the coefficients
\begin{eqnarray}\label{term4}
c_l = \sum\limits^{2^K-1}_{m=0} G\left(e^{i\frac{2\pi m}{2^K}}\right) e^{-il \frac{2\pi m}{2^K}}, l\in\mathbb{N}\cup\{0\}.
\end{eqnarray}

\noindent In fact, observing that function $e^{-i\frac{2\pi}{2^K}}$ has a period $2^K$, then coefficient $c_l$ is a periodic function of $2^K$, i.e., $c_l = c_{l+n2^K}, n\in\mathbb{N}$. Indeed, noticing $\left|a_je^{-i \frac{2\pi m}{2^K}}\right|<1,a_j\in\mathbb{D}$. This allows us to get another form of $(\ref{term81})$ as follows:
\begin{eqnarray*}\label{term1}
\langle G,e_{a_j}\rangle^{\tilde{~}}& = & \sum\limits^{2^K-1}_{m=0} \sum\limits^{+\infty}_{l=0} \frac{\sqrt{1- r^2}}{2^K}a_j^l G\left(e^{i\frac{2\pi m}{2^K}}\right) e^{-il \frac{2\pi m}{2^K}}
= \sum\limits^{+\infty}_{l=0}\frac{\sqrt{1- r^2}}{2^K}a_j^l \sum\limits^{2^K-1}_{m=0} G\left(e^{i\frac{2\pi m}{2^K}}\right) e^{-il \frac{2\pi m}{2^K}}\nonumber\\
& = & \sum\limits^{+\infty}_{s=0} \sum\limits^{(s+1)2^{K}-1}_{l=s2^K}\frac{\sqrt{1- r^2}}{2^K}a_j^l c_l
= \sum\limits^{2^K-1}_{l=0}\frac{\sqrt{1- r^2}}{2\pi}\frac{a_j^l}{1-a_j^{2^K}} c_l, a_j\in\mathbb{D}.
\end{eqnarray*}

\noindent Substituting $a_j = r e^{i\frac{2\pi}{2^K}j} \in \mathbb{D}$ in $(\ref{ku12})$, we have
\begin{eqnarray}\label{rcl}
&& \langle G,e_{a_j}\rangle^{\tilde{}}= \frac{\sqrt{1- r^2}}{2\pi\left(1-a_j^{2^K}\right)}\sum\limits^{2^K-1}_{l=0}a_j^l c_l
= \frac{\sqrt{1- r^2}}{2\pi\left(1 - \left(r e^{i\frac{2\pi}{2^K}j}\right)^{2^K}\right)}\sum\limits^{2^K-1}_{l=0}\left(r e^{i\frac{2\pi}{2^K}j}\right)^l c_l \\
&& = \frac{\sqrt{1- r^2}}{2\pi\left(1 - r^{2^K} e^{i\frac{2\pi}{2^K}j2^K}\right)}\sum\limits^{2^K-1}_{l=0}
r^l e^{i\frac{2\pi}{2^K}jl} c_l = \frac{\sqrt{1- r^2}}{2\pi\left(1 - r^{2^K}\right)}\sum\limits^{2^K-1}_{l=0}
r^l c_l e^{i\frac{2\pi}{2^K}jl},~j =0,1,\ldots, 2^K-1. \nonumber
\end{eqnarray}

\noindent To compute $(\ref{rcl})$, we divide it into two steps. First, applying the FFT to all of $c_l, l = 0,1,2,\ldots, 2^K-1$.
In fact, for arbitrary $0 \leq l \leq 2^K-1, l\in\mathbb{N}\cup\{0\}$, starting with $(\ref{term4})$, we have
\begin{eqnarray}\label{term5}
c_l & = & \sum\limits^{2^K-1}_{m=0} G\left(e^{i\frac{2\pi m}{2^K}}\right) e^{-il \frac{2\pi m}{2^K}}
= \sum\limits^{2^{K}-2}_{2m=0} G\left(e^{i\frac{2\pi (2m)}{2^K}}\right) e^{-il \frac{2\pi (2m)}{2^K}}
+ \sum\limits^{2^{K}-1}_{2m+1=1} G\left(e^{i\frac{2\pi (2m+1)}{2^K}}\right) e^{-il \frac{2\pi (2m+1)}{2^K}}\nonumber\\
& = & \sum\limits^{2^{K-1}-1}_{m=0} G\left(e^{i\frac{2\pi (2m)}{2^K}}\right) e^{-i \frac{2\pi (2m)}{2^K}l}
+ \sum\limits^{2^{K-1}-1}_{m=0} G\left(e^{i\frac{2\pi (2m+1)}{2^K}}\right) e^{-i \frac{2\pi (2m+1)}{2^K}l}.
\end{eqnarray}

\noindent Let $W_{2^{K}}=e^{-i\frac{2\pi}{2^{K}}}$, we get
\begin{eqnarray}\label{term3}
\left\{
\begin{array}{ll}
c_l = \sum\limits^{2^{K-1}-1}_{m=0} G\left(W^{-2m}_{2^K}\right) W^{2ml}_{2^K} + \sum\limits^{2^{K-1}-1}_{m=0} G\left(W^{-(2m+1)}_{2^K}\right) W^{2ml}_{2^K} W^l_{2^K}, \\
c_{l+2^{K-1}} = \sum\limits^{2^{K-1}-1}_{m=0} G\left(W^{-2m}_{2^K}\right) W^{2ml}_{2^K}
- \sum\limits^{2^{K-1}-1}_{m=0} G\left(W^{-(2m+1)}_{2^K}\right) W^{2ml}_{2^K} W^l_{2^K}.
\end{array}
\right.
\end{eqnarray}

\noindent The second step is to use FFT formulation and to derive the following theorem.

\begin{theorem}\label{the2.1}
The right hand side of $(\ref{term81})$ is equal to the cases
\begin{eqnarray}\label{term8}
\left\{
\begin{array}{ll}
\langle G,e_{a_j}\rangle^{\tilde{~}} = \frac{\sqrt{1- r^2}}{2^K \left(1 - r^{2^K} \right)}\left( \sum\limits^{2^{K-1}-1}_{l=0}
r^{2l}c_{2l} W_{2^{K}}^{-2jl} + \sum\limits^{2^{K-1}-1}_{l=0}
r^{2l+1} c_{2l+1} W_{2^{K}}^{-2jl}W_{2^{K}}^{-j}\right),\\
\langle G,e_{a_{j+2^{K-1}}}\rangle^{\tilde{~}}
= \frac{\sqrt{1- r^2}}{2\pi\left(1 - r^{2^K} \right)}\left( \sum\limits^{2^{K-1}-1}_{l=0}
r^{2l}c_{2l} W_{2^{K}}^{-2jl} - \sum\limits^{2^{K-1}-1}_{l=0}
r^{2l+1} c_{2l+1} W_{2^{K}}^{-2jl}W_{2^{K}}^{-j}\right),
\end{array}
\right.
\end{eqnarray}
where $j=0,1,2,\ldots, 2^{K-1}-1$ and $c_l, l =0,1,2,\ldots, 2^{K-1}-1$, is given by $(\ref{term3})$.
\end{theorem}

\begin{proof}
Observing from $(\ref{term81})$, one gets
\begin{eqnarray}\label{term82}
&& \langle G,e_{a_j}\rangle^{\tilde{~}}=\frac{\sqrt{1- r^2}}{2^K \left(1-a_j^{2^K}\right)}\sum\limits^{2^K-1}_{l=0}a_j^l c_l
= \frac{\sqrt{1- r^2}}{2^K \left(1 - \left(r W_{2^{K}}^{-j}\right)^{2^K}\right)}\sum\limits^{2^K-1}_{l=0}\left(r W_{2^{K}}^{-j}\right)^l c_l \\
&& = \frac{\sqrt{1- r^2}}{2^K \left(1 - r^{2^K} W_{2^{K}}^{-j2^K}\right)}\sum\limits^{2^K-1}_{l=0}
r^l W_{2^{K}}^{-jl} c_l = \frac{\sqrt{1- r^2}}{2^K \left(1 - r^{2^K}\right)}\sum\limits^{2^K-1}_{l=0}
r^l c_l W_{2^{K}}^{-jl} \nonumber\\
&& = \frac{\sqrt{1- r^2}}{2^K \left(1 - r^{2^K} \right)}\sum\limits^{2^K-2}_{2l=0}
r^{2l} W_{2^{K}}^{-2lj} c_{2l} + \frac{\sqrt{1- r^2}}{2^K \left(1 - r^{2^K}\right)}\sum\limits^{2^K-1}_{2l+1=1}
r^{2l+1} W_{2^{K}}^{-2lj}W_{2^{K}}^{-j} c_{2l+1} \nonumber\\
&& = \frac{\sqrt{1- r^2}}{2^K \left(1 - r^{2^K} \right)}\left( \sum\limits^{2^{K-1}-1}_{l=0}
r^{2l} W_{2^{K}}^{-2lj} c_{2l} + \sum\limits^{2^{K-1}-1}_{l=0}
r^{2l+1} W_{2^{K}}^{-2lj}W_{2^{K}}^{-j}c_{2l+1} \right), j =0,1,\ldots, 2^K-1.
\end{eqnarray}

\noindent Hence, we get
\begin{eqnarray}
\langle G,e_{a_{j+2^{K-1}}}\rangle^{\tilde{~}}&=& \frac{\sqrt{1- r^2}}{2^K \left(1 - r^{2^K} \right)}\left( \sum\limits^{2^{K-1}-1}_{l=0}
r^{2l}c_{2l} W_{2^{K}}^{-2\left(j+2^{K-1}\right)l} + \sum\limits^{2^{K-1}-1}_{l=0}
r^{2l+1} c_{2l+1} W_{2^{K}}^{-2\left(j+2^{K-1}\right)l}W_{2^{K}}^{-\left(j+2^{K-1}\right)}\right)\nonumber\\
&=& \frac{\sqrt{1- r^2}}{2^K \left(1 - r^{2^K} \right)}\left( \sum\limits^{2^{K-1}-1}_{l=0}
r^{2l}c_{2l}W_{2^{K}}^{-2jl} - \sum\limits^{2^{K-1}-1}_{l=0}
r^{2l+1} c_{2l+1} W_{2^{K}}^{-2jl}W_{2^{K}}^{-j}\right).
\end{eqnarray}

\noindent Therefore, for $0 \leq j < 2^{K-1}-1$, we have
\begin{eqnarray}
\left\{
\begin{array}{ll}
\langle G,e_{a_j}\rangle^{\tilde{~}}= \frac{\sqrt{1- r^2}}{2^K \left(1 - r^{2^K} \right)}\left( \sum\limits^{2^{K-1}-1}_{l=0}
r^{2l}c_{2l} W_{2^{K}}^{-2jl} + \sum\limits^{2^{K-1}-1}_{l=0}
r^{2l+1} c_{2l+1} W_{2^{K}}^{-2jl}W_{2^{K}}^{-j}\right), \\
\langle G,e_{a_{j+2^{K-1}}}\rangle^{\tilde{~}}= \frac{\sqrt{1- r^2}}{2^K \left(1 - r^{2^K} \right)}\left( \sum\limits^{2^{K-1}-1}_{l=0}
r^{2l}c_{2l} W_{2^{K}}^{-2jl} - \sum\limits^{2^{K-1}-1}_{l=0}
r^{2l+1} c_{2l+1} W_{2^{K}}^{-2jl}W_{2^{K}}^{-j}\right).
\end{array}
\right.
\end{eqnarray}
\end{proof}

\begin{remark}

\noindent Since the computational complexity of directly computing (\ref{term81}) is $\mathcal{O}(N^2)$, it is unacceptable especially when $N$ takes a large positive integer. Thus, what is the key point is to reduce the computational complexity of (\ref{term81}) from $\mathcal{O}(N^2)$ to $\mathcal{O}(N\log_2 N)$. We achieve this by constructively changing (\ref{term81}) into (\ref{ku12}) with coefficients $c_l$ given by (\ref{term4}), and then transferring (\ref{ku12}) into (\ref{term8}). This is because through the technical observation we can make full use of the fast Fourier transform algorithm (FFT) to compute (\ref{term8}) and coefficients $c_l$ given by (\ref{term4}), whose computational complexity is $\mathcal{O}(N\log_2 N)$. These ideas are the starting of the following algorithm.

\end{remark}

\subsection{Algorithm}\label{algorithm}

\noindent In this section, we propose our fast algorithm for the 1-D AFD, making use of the FFT mechanism. Parameters are expressed in polar coordinate. Radius is sampled evenly by $M$ points as $0 <r_1 < r_2 <\ldots <r_s<\ldots<r_{M-1}< r_M <1$, where $r_s = \frac{s}{M+1}, s = 1,2,\ldots,M $.  Denote the grid mesh on the circle of radius $r_s$ is
$$C_{s,2^{K}}=\{a_j = r_s W_{2^{K}}^{-j}, j = 0, 1, 2, \dots, 2^{K}-1\}.$$ This set is in particular involved in the FFT formulation of the subsection $3.1$.

\subsubsection{Algorithm description}

\noindent {\bf Step 1}. Compute $\langle G,e_{a_j}\rangle^{\tilde{}}, j =0,1,2,\ldots, 2^{K}-1$ for all of $a_j\in C_{s,2^{K}}$, $s=1,2,\dots,M$.\\
\noindent Applying $(\ref{term3})$, we compute all of $r_s^l, c_l, l = 0,1,2,\ldots, 2^K-1$, $s=1,2,\dots,M$, and store them.\\ Next, starting from the recursive relation $(\ref{term8})$ with a fixed $r_s$, we need to work backward $K$ times to obtain the expression of the input data. Set the input of the recursive relation to be $f(l)=r_s^lc_l,l=1,2,\dots,2^K-1$. Then the twiddle factor of $(\ref{term8})$ is $W_{2^{K}}^l$, while that of FFT is $W_{2^{K}}^{-l}$. Because $f$ is output in the order of $l$, the bit-reversal permutation of $m$ is necessary. Similarly, we need to input $f$ in the order of the bit-reversal permutation of $l$ to obtain $\langle G,e_{a_j}\rangle^{\tilde{}}, j =0,1,2,\ldots, 2^{K}-1$. We repeat the above procedure to obtain $\langle G,e_{a_j}\rangle^{\tilde{}}$ for all $s=1,2,\dots,M$.

\noindent {\bf Step 2.} Find a point $a_{s',j'}\in\displaystyle \mathop{\cup}_{s=1}^M C_{s,2^{K}}$ satisfying $\left|\langle G,e_{a_{s',j'}}\rangle^{\tilde{}}\right|^2 = \max\limits_{\substack{a_j\in C_{s,2^{K}}\\s=1,2,\dots,M }}\left|\langle G,e_{a_j}\rangle^{\tilde{}}\right|^2$.\\
We compare all $\langle G,e_{a_j}\rangle^{\tilde{}}$, $j =0,1,2,\ldots, 2^{K}-1, s= 1,2,\ldots,M$, to find the maximum value and the corresponding parameter $a_{s',j'}=r_{s'} W_{2^{K}}^{-j'}$.

\begin{remark}

In our algorithm proposed in the subsection $3.1$, we select evenly the sampling. While the non-uniform/non-evenly sampling is allowed, our algorithm still holds by applying the non-uniform discrete Fourier transforms developed in Ref. e.g. $\cite{FJ}$.

\end{remark}

\subsection{Computational complexity}

\noindent In this section we will analyse the computation complexity of the proposed algorithm in Section $\ref{algorithm}$. Here let $N=2^K$.

\noindent In Step 1, applying the argument of the classical FFT, the computational complexity of $(\ref{term3})$ is $\mathcal{O}(N\log_2 N)$. The computational complexity of the input signal $f(l)=r^l c_l, l = 0,1,2,\ldots, N-1$ is $\mathcal{O}(N+N\log_2 N)$, which is also $\mathcal{O}\left(N \log_2 N\right)$. For $s =1,2,\ldots,M$, the total computational complexity of $\langle G,e_{a_j}\rangle^{\tilde{}},j =0,1,\ldots, 2^K-1, s=1,2,\ldots,M$, is $\mathcal{O}\left(MN \log_2 N\right)$.

\noindent In Step 2, in order to find a point $a_{s',j'}\in\displaystyle \mathop{\cup}_{s=1}^M C_{s,2^{K}}$, satisfying $\left|\langle G,e_{a_{s',j'}}\rangle^{\tilde{}} \right|^2 =  \max\limits_{\substack{a_j\in C_{s,2^{K}}\\s=1,2,\dots,M }}\left|\langle G,e_{a_j}\rangle^{\tilde{}}\right|^2$, we compare the absolute value of all $\langle G,e_{a_j}\rangle^{\tilde{}}$ obtained from Step 1 one by one. As the length of the above absolute values is $MN$, the computational complexity to find the maximum is $\mathcal{O}(MN)$.

\noindent Totally, the computational complexity of Algorithm 1 is $\mathcal{O}\left(MN \log_2 N\right)$.

\begin{remark}
For a fixed $r:0<r<1$, $\langle G,e_{a_j}\rangle^{\tilde{~}}, |a_j|=r, ~j = 0,1,2,\ldots, 2^{K}-1$, could be computed directly from $(\ref{term81})$, the computational complexity is $\mathcal{O}(N^2)$, being considerably higher than that of $(\ref{term3})$.
But if we do not chose evenly $2^{K-m}\left(0 < m <K, m\in\mathbb{N}\right)$ on the circle of radius $r$, the argument of the classical FFT cannot play a role in the computation of 1-D AFD. In such a case, $\langle G,e_{a_j}\rangle^{\tilde{}}$ can be obtained from $(\ref{term81})$ with the computational complexity $\mathcal{O}(N^2)$.
\end{remark}

\section{Numerical experiments}

\noindent In this section, we temporarily call the proposed algorithm FFT-AFD. We also call the direct computation of AFD without FFT mechanism as Direct-AFD. The efficiency of the proposed algorithm is analysed from two different aspects. The first one is to list the output of the proposed algorithm: the running times, the approximation results, the parameters $a_k$ and relative errors. The second one is to make a running time comparison when the original signal is discretized by a different length of samples. The comparisons are presented between the output of the proposed algorithm and the ones of the Direct-AFD algorithm in \cite{TLZ}.

\noindent We define the relative error by
\begin{eqnarray}\label{er}
\delta=\frac{||G-S_n||^2}{||G||^2},
\end{eqnarray}
where $G$ is the original signal and $S_n$ is the summation of $n$ terms, i.e. $S_n=\sum\limits^n_{k=1}\left<G,B_k\right>B_{k}$.

\noindent In all of the following experiments, the radius of the unit disc $r=0,0.1,0.2,...0.8$ will be considered. The CPU of the utilized computer is the Intel G540 under the default setting of a single thread. All experiments are conducted in Matlab 2012b.

\subsection{Basic experiments}

\noindent All original functions are sampled by the 1024 points in this part.

\noindent {\bf Case 1.} The original function is chosen as $$f_1=\frac{(0.0247e^{i3t}+0.355e^{i2t})}{(1-0.3679e^{it})}\in H^2.$$ 

\noindent The time consuming to run 10 steps is shown in Table 1 below. The experiments are repeated 6 times.

\begin{table}[!th]
\caption{Running time (s)}
\label{table2}
\centering
\begin{tabular}{|c|c|c|c|c|c|c|}
\hline
FFT-AFD&0.2617& 0.2609&0.2613& 0.2610& 0.2609& 0.2614 \\
\hline
Direct-AFD& 1.3315&1.3297 & 1.3294& 1.3323& 1.3294 & 1.3329 \\
\hline 
\end{tabular}
\end{table}

\noindent The real part of approximation results is shown in Table \ref{c1}.
\begin{table}[!th]
\centering
\caption{Comparison between the approximation results in Case 1}
\label{c1}
\begin{tabular}{ccccc}
\hline
$S_2$&$S_4$&$S_6$&$S_{8}$&$S_{10}$\\
\hline
\noalign{\vskip 1mm}
\includegraphics[width=1.3in]{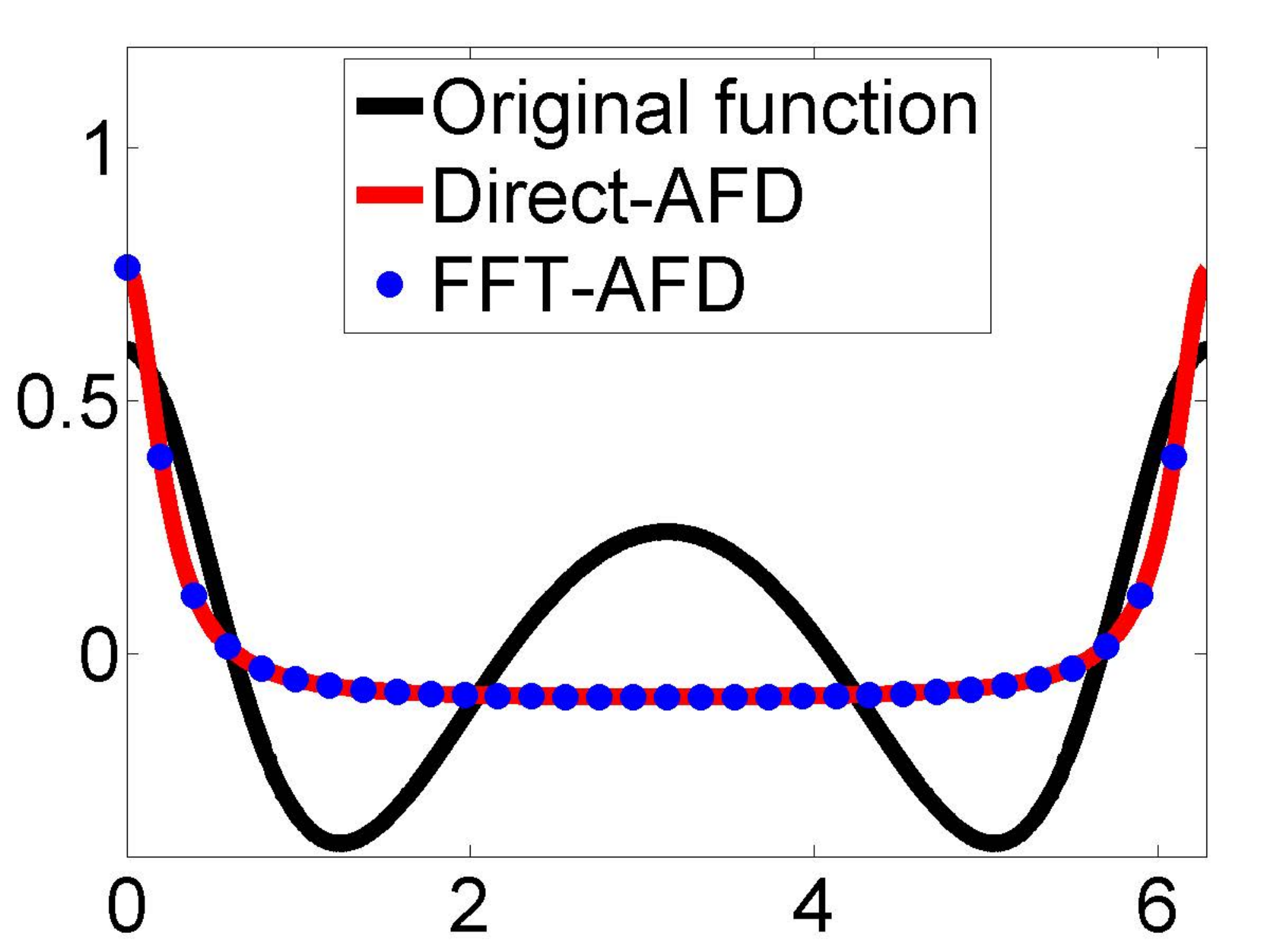}&\includegraphics[width=1.3in]{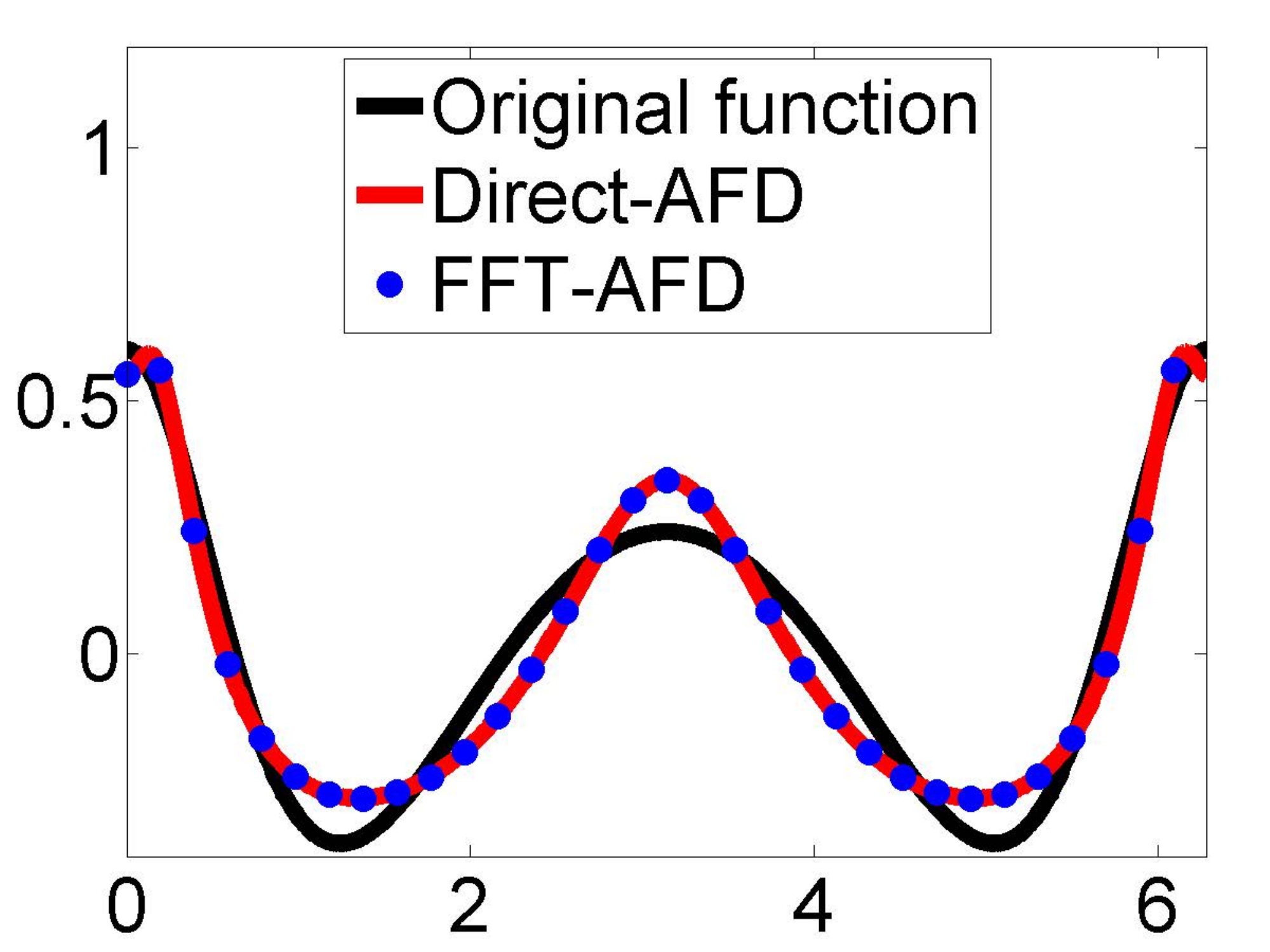}&\includegraphics[width=1.3in]{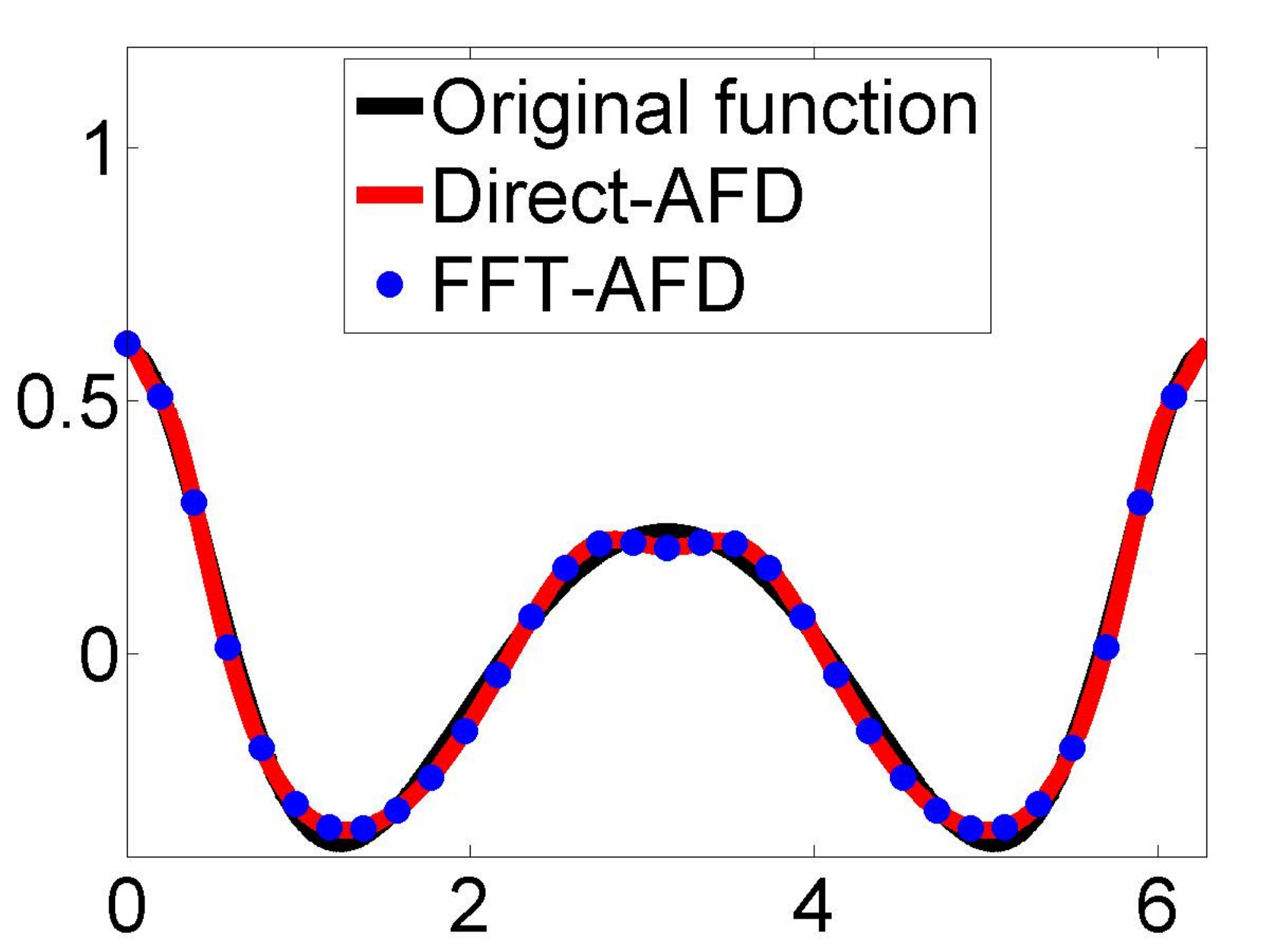}&\includegraphics[width=1.3in]{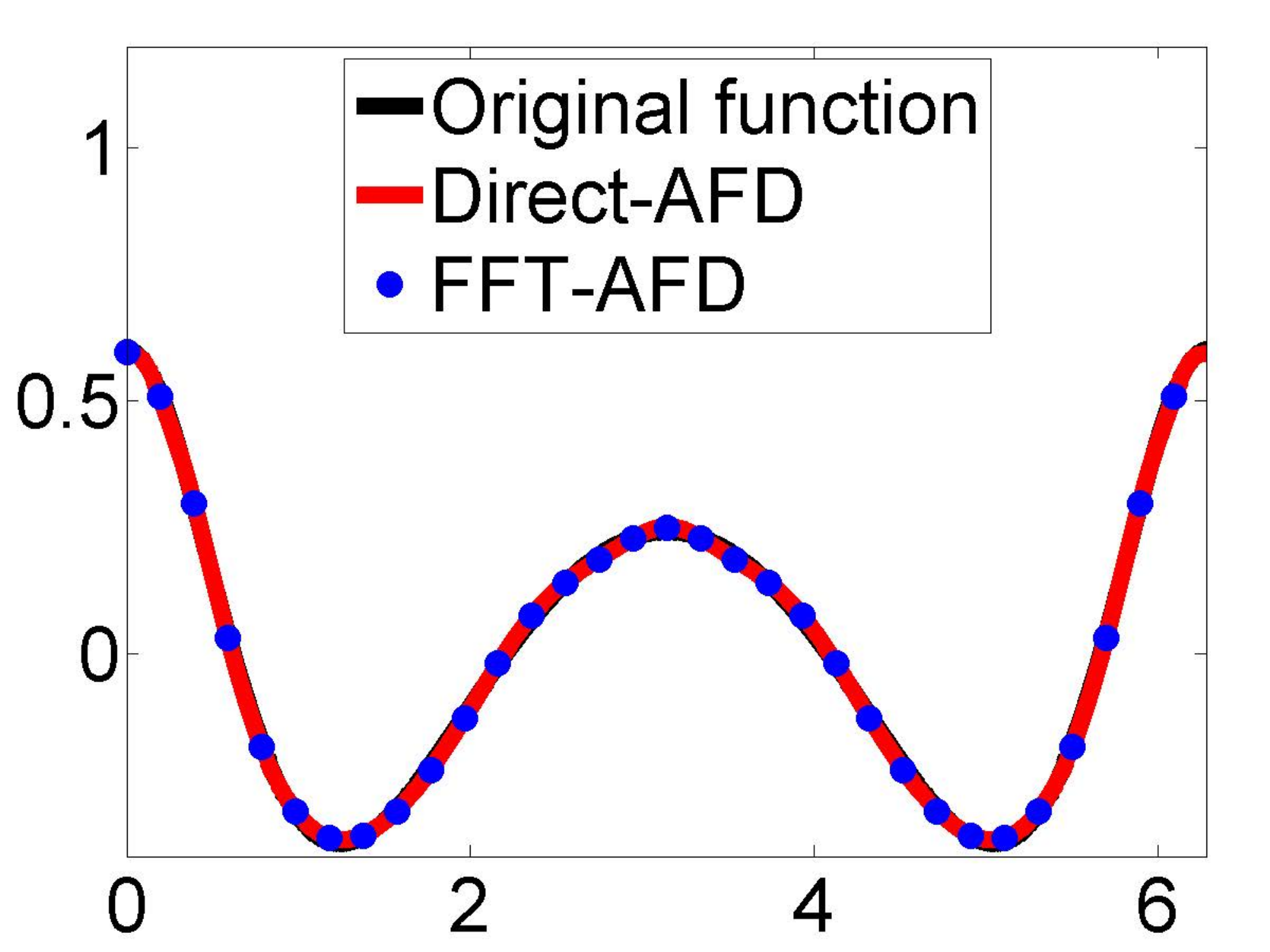}&\includegraphics[width=1.3in]{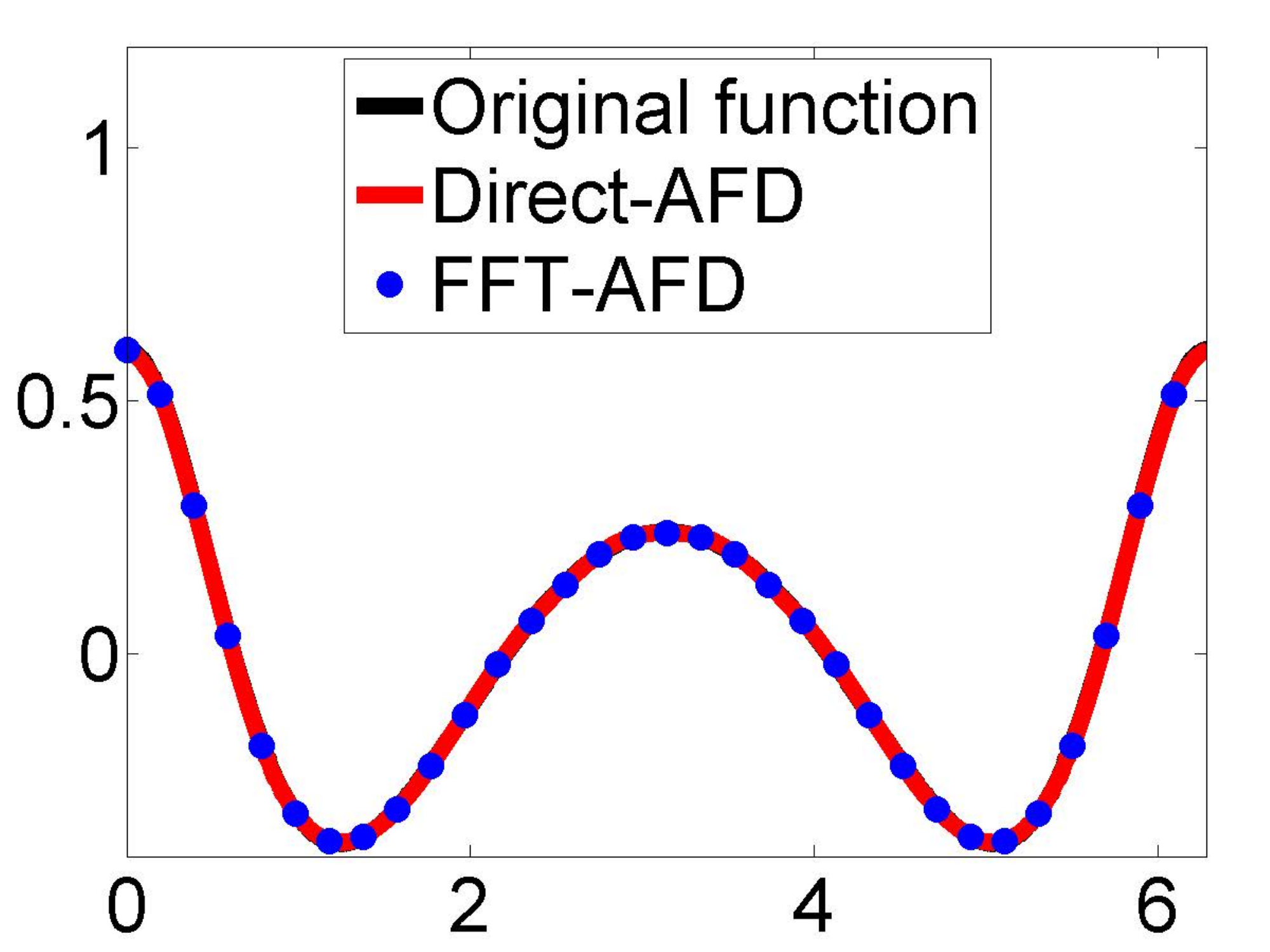}\\
\hline
\end{tabular}
\end{table}

\noindent The $a_k$'s and the relative errors obtained from 1-D AFD with our proposed algorithm are shown in the following Tables \ref{table1} and \ref{table3}.

\begin{table}[!th]
\begin{minipage}[t]{.5\textwidth}
\caption{Parameters $a_k$ in Case 1}
\includegraphics[width=4in]{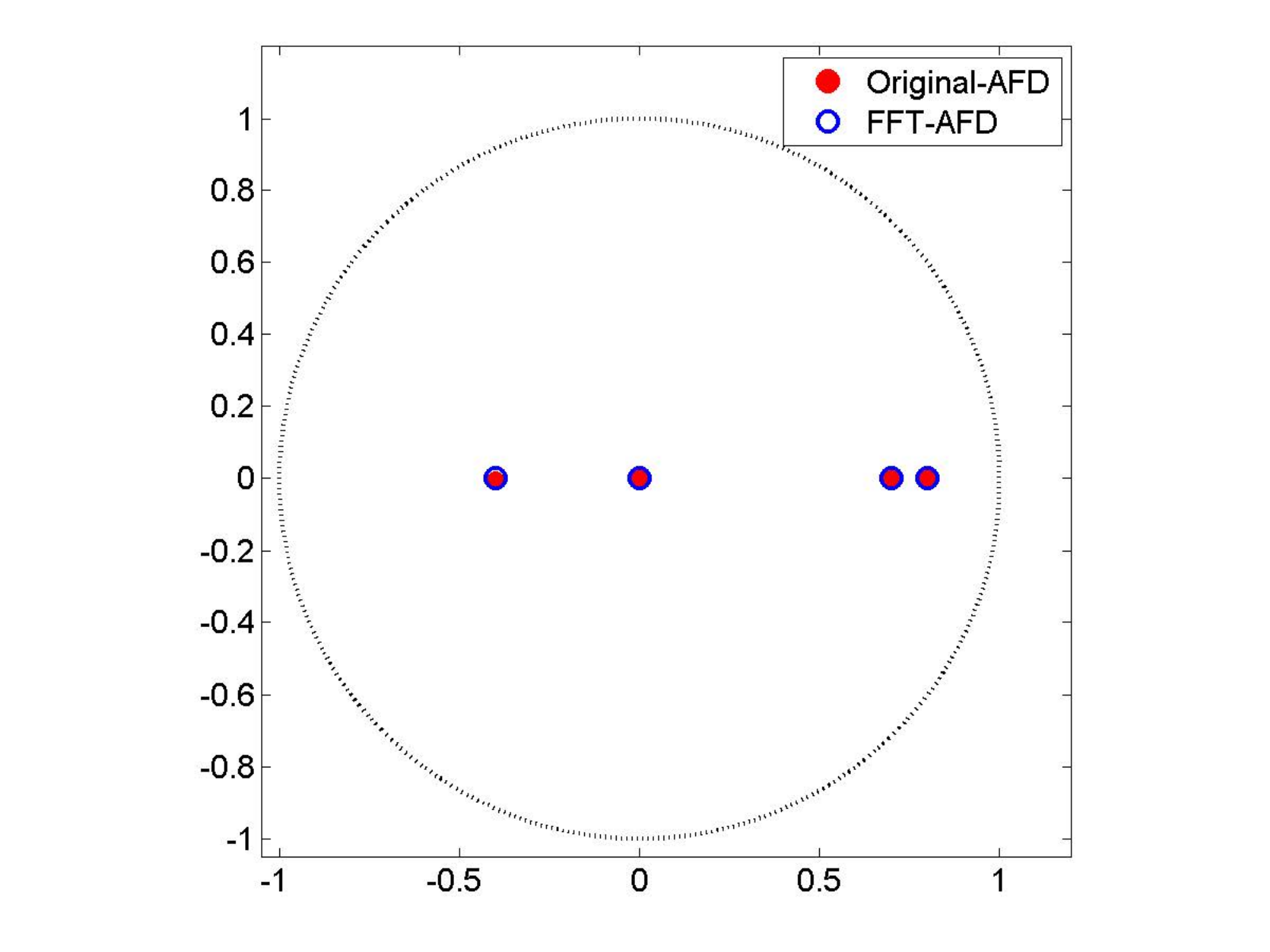}
\label{table1}
\end{minipage}
\begin{minipage}[t]{.5\textwidth}
\caption{Relative error (dB) }
\label{table3}
\centering
\begin{tabular}{|c|c|c|c|c|}
\hline 
$N$&FFT-AFD&Direct-AFD\\
\hline
$ 1$ & 1.0000   & 1.0000 \\
\hline 
2 & 0.5790   & 0.5790 \\
\hline 
3 & 0.2092 & 0.2093 \\
\hline 
4 &0.0553  & 0.0553 \\
\hline 
5 & 0.0189  & 0.0189 \\
\hline 
6 & 0.0052 & 0.0052 \\
\hline 
7 & 0.0017 & 0.0017 \\
\hline 
8 & 0.0005& 0.0005 \\
\hline 
9 & 0.0002& 0.0002 \\
\hline 
10 & 0.0000 & 0.0001 \\
\hline 
\end{tabular}
\end{minipage}
\end{table}

\noindent {\bf Analysis 1.} From Table \ref{table2}, it is clear that the numerical computation is significantly accelerated. From Table \ref{table1}, the difference of the parameters $a_k$ and that of the relative error are limited in 0.0001. These data indicate that our proposed algorithm actually achieves the effects of 1-D AFD in less time.

\noindent {\bf Case 2.} $f_1$ is a rational function being of good smoothness. Then $f_2$ is given as an example of the step functions as $$f_2=sgn(\sin{t}).$$ 

\noindent The running time, $a_k$ and the relative error are given in Tables $\ref{table4}$ as follows.

\begin{table}[h]
\caption{Running time (s)}
\label{table4}
\centering
\begin{tabular}{|c|c|c|c|c|c|c|}
\hline
FFT-AFD&0.2618& 0.2616&0.2593& 0.2598&0.2598&0.2598 \\
\hline
Direct-AFD&1.3110& 1.3157& 1.3294& 1.3099&1.3150 & 1.3079\\
\hline 
\end{tabular}
\end{table}

\noindent The approximation results is shown in Table \ref{c2}.
\begin{table*}[h]
\centering
\caption{Comparison between the approximation results in Case 2}
\label{c2}
\begin{tabular}{ccccc}
\hline
$S_2$&$S_4$&$S_6$&$S_{8}$&$S_{10}$\\
\hline
\noalign{\vskip 1mm}
\includegraphics[width=1.3in]{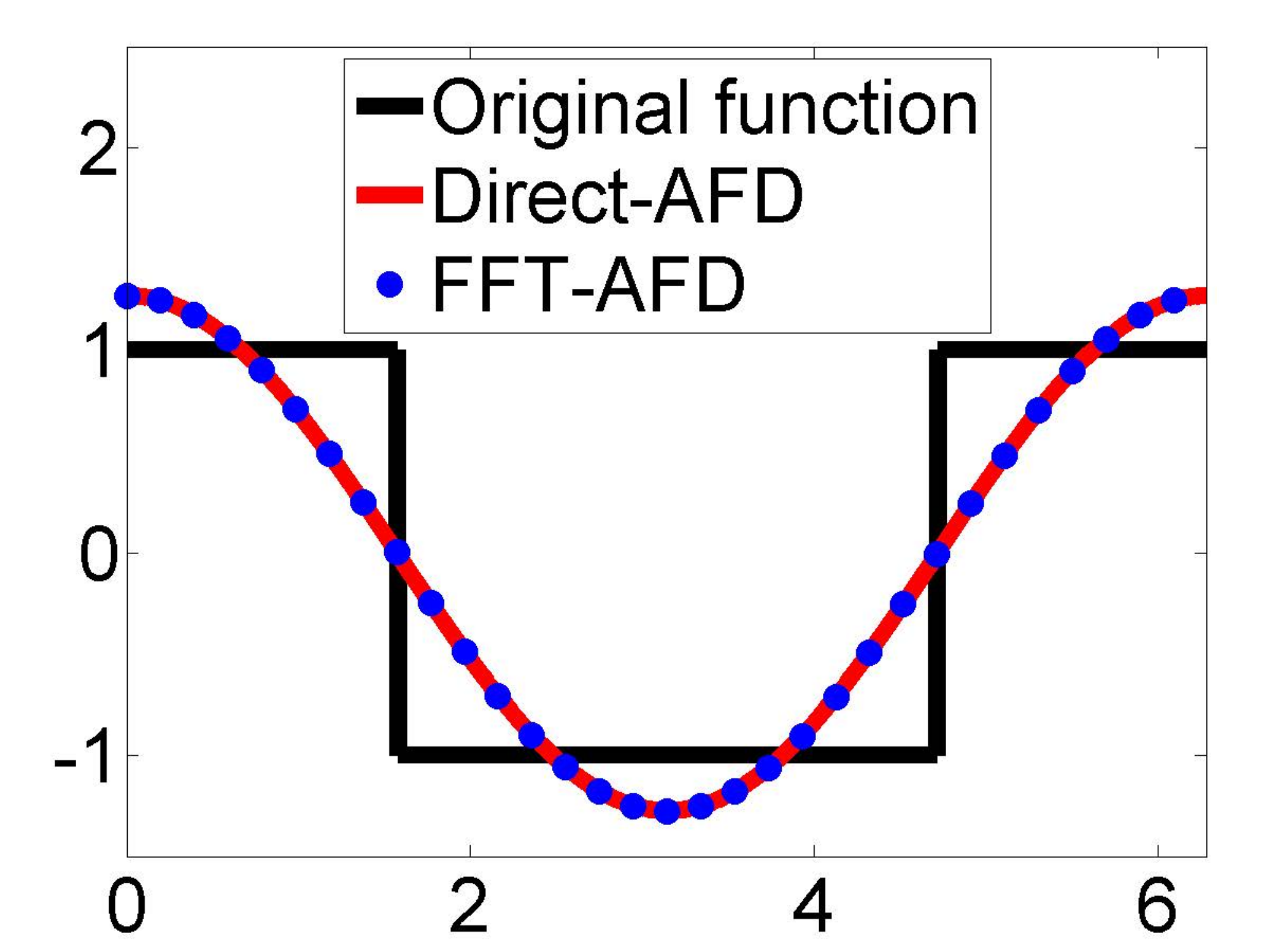}&\includegraphics[width=1.3in]{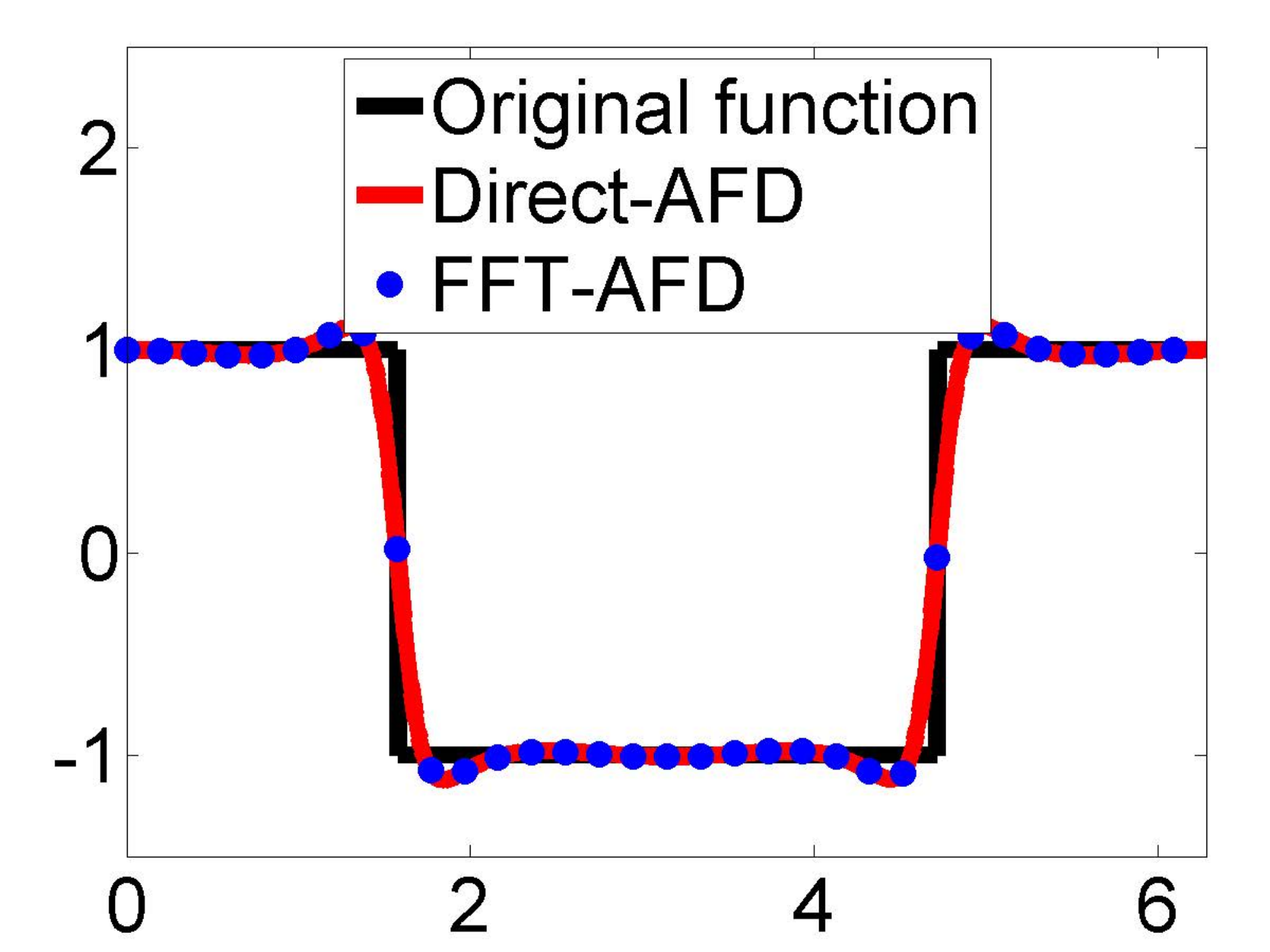}&\includegraphics[width=1.3in]{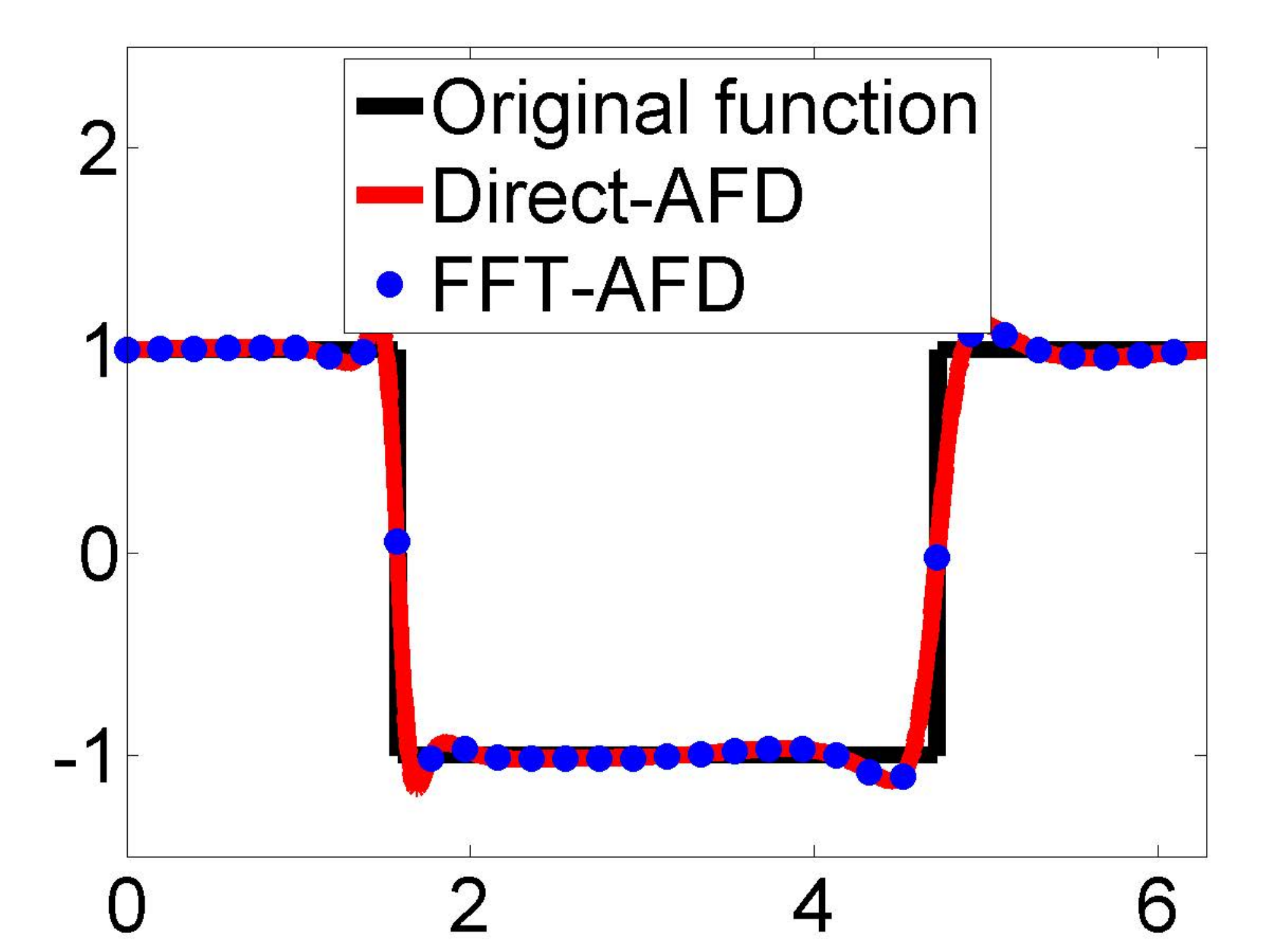}&\includegraphics[width=1.3in]{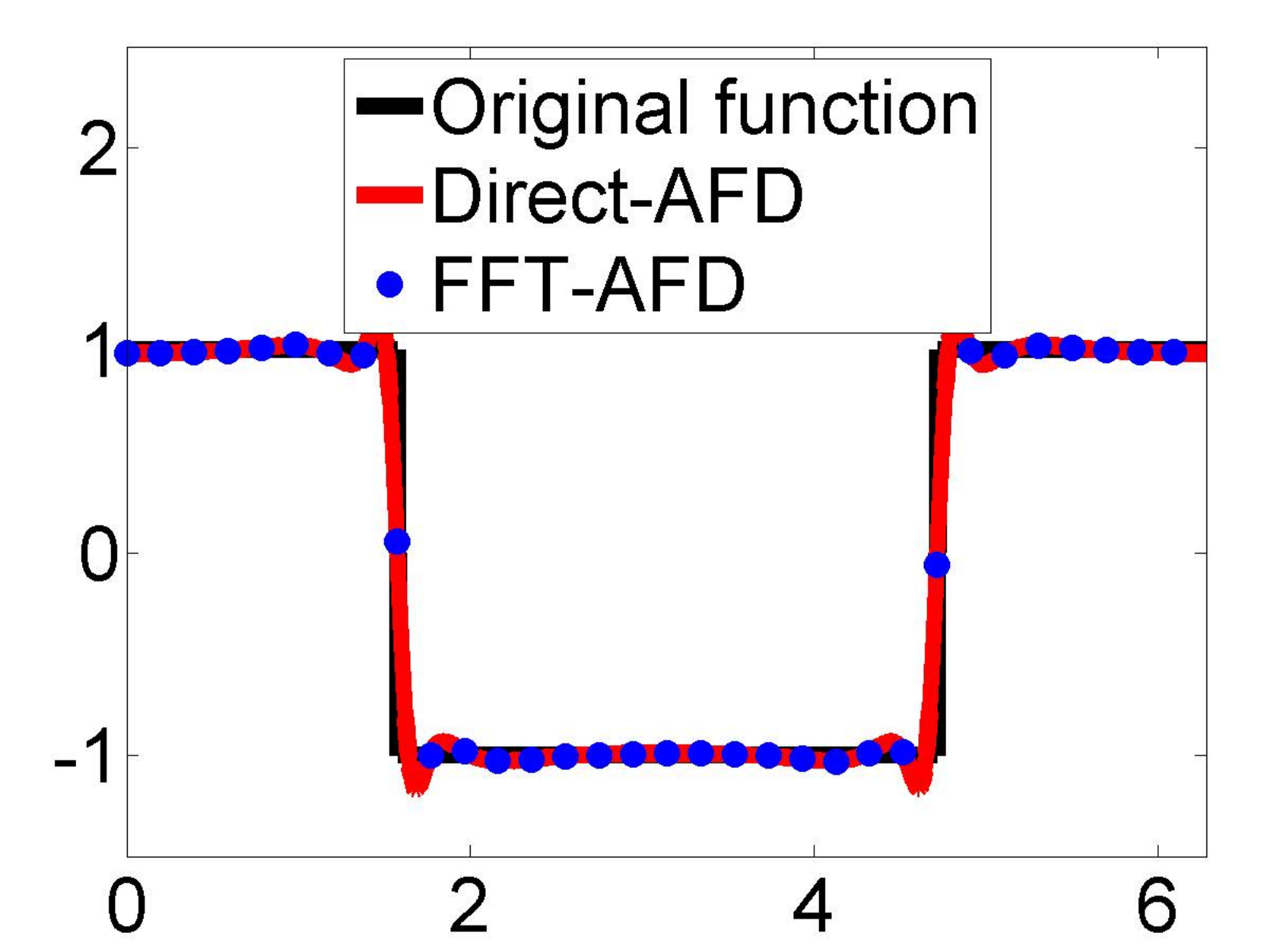}&\includegraphics[width=1.3in]{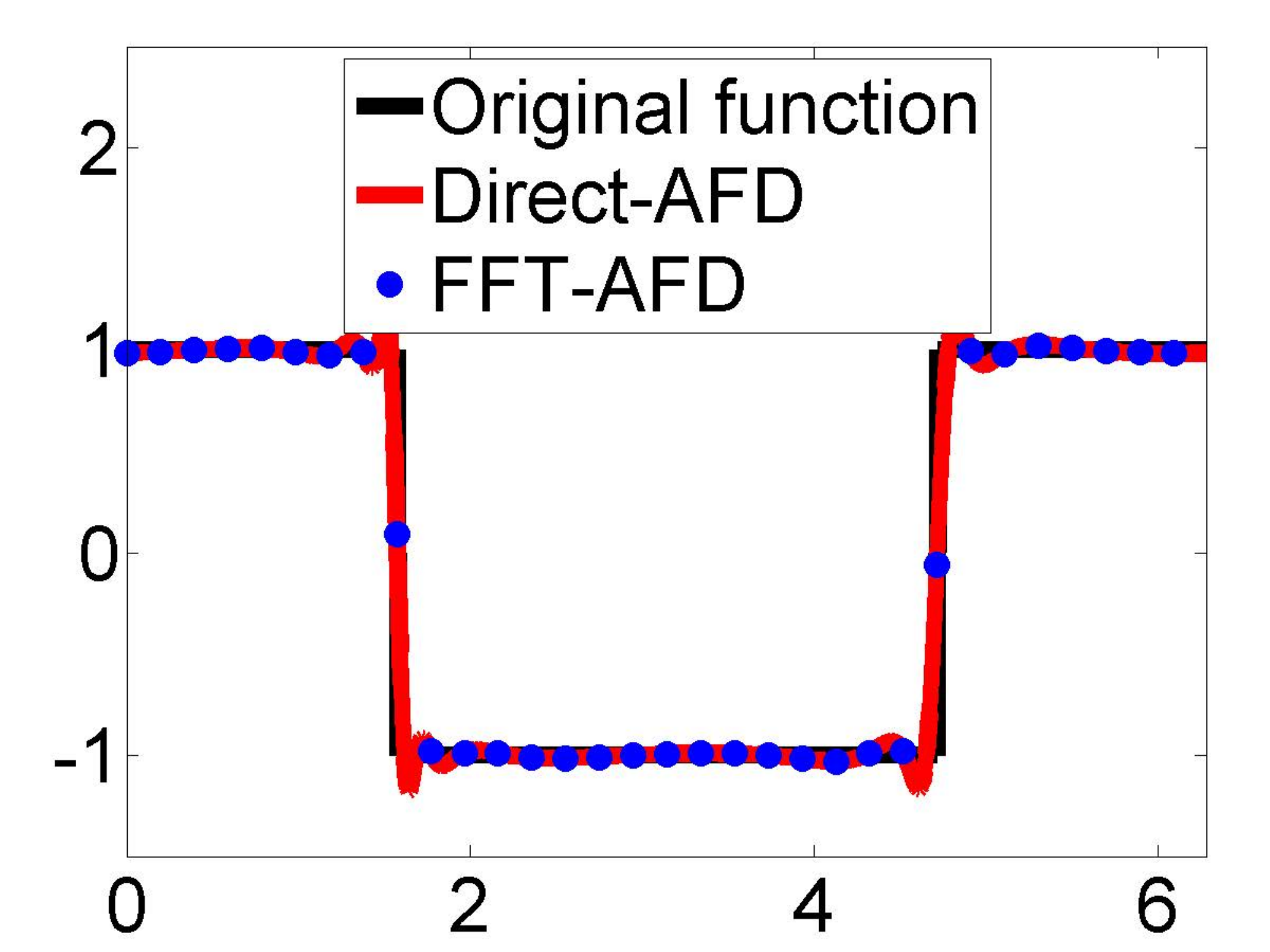}\\
\hline
\end{tabular}
\end{table*}

\noindent The $a_k$'s and the relative errors obtained from our proposed algorithm are shown in Tables \ref{table5} and \ref{table6}.

\begin{table}[h]
\begin{minipage}[t]{.5\textwidth}
\centering
\caption{Parameters $a_k$ in Case 1}
\includegraphics[width=4in]{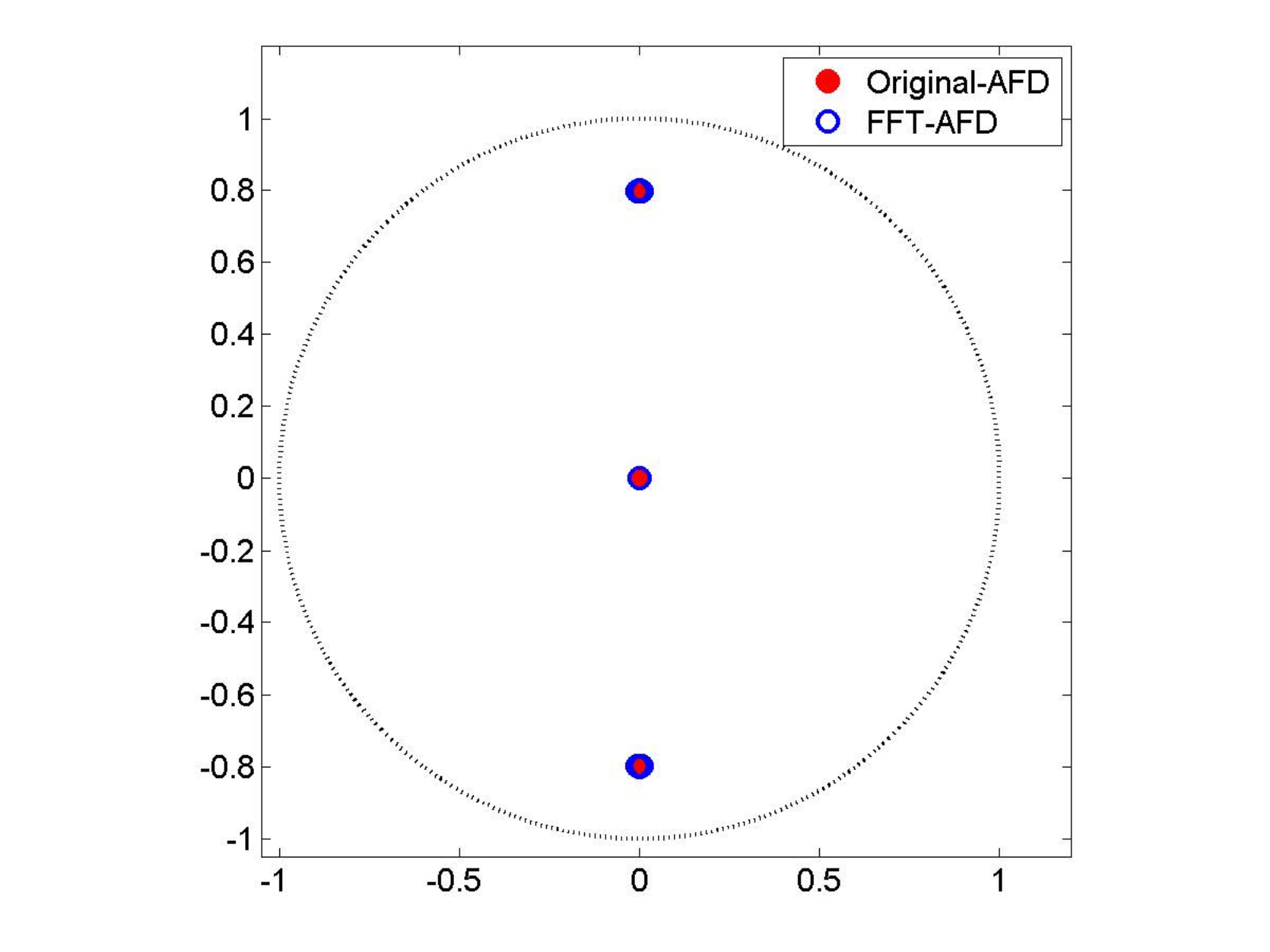}
\label{table5}
\end{minipage}
\begin{minipage}[t]{.5\textwidth}
\caption{Relative error (dB) }
\label{table6}
\centering
\begin{tabular}{|c|c|c|c|c|}
\hline 
$N$&FFT-AFD&Direct-AFD\\
\hline
1& 1.0000  & 1.0000\\
\hline 
2& 0.1895& 0.1895\\
\hline 
3& 0.1260  & 0.1260\\
\hline 
4& 0.0266 & 0.0266\\
\hline 
5& 0.0247& 0.0247\\
\hline 
6& 0.0199  & 0.0199\\
\hline 
7& 0.0183 & 0.0183\\
\hline 
8& 0.0129  & 0.0129\\
\hline 
9& 0.0120 & 0.0120\\
\hline 
10& 0.0106  & 0.0105 \\
\hline 
\end{tabular}
\end{minipage}
\end{table}

\noindent {\bf Analysis 2.} The running time of our proposed algorithm keeps a stable acceleration. The parameters $a_k$'s have a little difference from those of the Direct-AFD, in which the value of the integral (inner product) is yielded by the Newton-Cotes rules. It does not cause any influence on the relative errors in this example.

\subsection{The Influence of the length of samples}

\noindent The proposed algorithm in Section $3$ is of $\mathcal{O}\left(MN \log_2 N\right)$ computation complexity. Compared to Direct-AFD with the computational complexity $\mathcal{O}\left(MN^2\right)$, the running time should be obviously improved with an increased length of samples of the original function. Hereby, the original signal $f_1$ is sampled by $128, 256, 512, 1024, 2048$ points respectively. We observe the running times of the proposed algorithm approximating 10 steps with the sampled signals. The running time can be found in Table \ref{table10}.
\begin{table}[h]
\caption{Running time (s) }
\label{table10}
\centering
\begin{tabular}{|c|c|c|c|c|c|c|}
\hline
The length of samples&128	&256&	512&	1024&	2048&	4096\\
\hline 
FFT-AFD&0.0331&	0.0634&	0.1325&	0.2662&	0.5416&	1.1243\\
\hline 
Direct-AFD&0.0369&	0.1049&	0.3530&	1.3087&	5.1891&	24.3431\\
\hline 
\end{tabular}
\end{table}

\noindent The parameters have a little difference from those from Direct-AFD, when $N = 2^7, 2^8, 2^9, 2^{10}, 2^{11}, 2^{12}$, respectively. Moreover, it does not cause any influence on the relative errors in this example.

\begin{remark}

In our algorithm presented in the subsections $3.2$ and the numerical experiments in Section $4$, we have considered the case of the dynamic parameter choice $N = 2^K$. 
Moreover, following the argument contained in Ref. e.g. $\cite{CM}$, our algorithm is still valid when the dynamic parameter choice $N = p^K$ with $p$ being an arbitrary prime is considered.

\end{remark}



\end{document}